\documentclass[12pt,twoside]{amsart}
\usepackage{amssymb,amsmath,amsthm}
\usepackage{verbatim}
\usepackage{graphicx}
\usepackage{epsfig}
\usepackage{color, soul}
\usepackage[all]{xy}

\DeclareMathAlphabet{\mathpzc}{OT1}{pzc}{m}{it}

\newtheorem{theorem}{Theorem}[section]
\newtheorem{lemma}[theorem]{Lemma}
\newtheorem{definition}[theorem]{Definition}
\newtheorem{corollary}[theorem]{Corollary}
\newtheorem{proposition}[theorem]{Proposition}
\newtheorem{remark}[theorem]{Remark}

\newtheorem{result}[theorem]{Result}

\def\K{\mathcal K}

\def\d{\displaystyle}

\def\pco{\text{{\it p}{\rm -co}}}
\def\co{{\rm co}}
\def\eps{\varepsilon}
\def\m{\mathpzc{m}}
\def\ker{{\rm ker\, }}
\def\v{\mathpzc{v}}

\title[]{The ideal of $p$-compact operators: \\ a tensor product approach}

\author{Daniel Galicer, Silvia Lassalle and Pablo Turco}

\thanks{This project was supported in part by UBACyT X038, UBACyT X218  and CONICET PIP 0624}

\address{Departamento de Matem\'{a}tica - Pab I,
Facultad de Cs. Exactas y Naturales, Universidad de Buenos Aires,
(1428) Buenos Aires, Argentina}

\email{dgalicer@dm.uba.ar, slassall@dm.uba.ar, paturco@dm.uba.ar}

\begin{document}

\begin{abstract} We study the space of $p$-compact operators $\mathcal K_p$,  using the theory of tensor norms and operator ideals. We prove that  $\mathcal K_p$ is associated to $/d_p$, the left injective associate of the Chevet-Saphar tensor norm $d_p$ (which is equal to $g_{p'}'$). This allows us to relate the theory of $p$-summing operators with that of $p$-compact operators. With the results known for the former class and appropriate hypothesis on $E$ and $F$ we prove that $\mathcal K_p(E;F)$ is equal to  $\mathcal K_q(E;F)$ for a wide range of values of $p$ and $q$, and show that our results are sharp. We also exhibit several structural properties of $\mathcal K_p$. For instance, we obtain that $\mathcal K_p$ is regular, surjective, totally accessible and characterize its maximal hull $\mathcal K_p^{max}$ as the dual ideal of the $p$-summing operators, $\Pi_p^{dual}$. Furthermore, we prove that  $\mathcal K_p$ coincides isometrically with $\mathcal {QN}_p^{dual}$, the dual ideal of the quasi $p$-nuclear operators.  
\end{abstract}
\maketitle
\parindent 10pt
\parskip .2cm

\section*{Introduction}
In 1956, Grothendieck published his famous  Resume \cite{Gro}  in which he set out a complete theory of tensor products of Banach spaces. In the years following, the parallel theory of operator ideals was initiated by Pietsch \cite{PIE}. Researchers in the field have generally preferred the language of operator ideals to the more abstruse language of tensor products, and so the former theory has received more attention in the succeeding decades. However, the monograph of Defant and Floret \cite{DF}, in which the two fields are described in tandem, has initiated a period in which authors use indistinctly both languages.

In the recent years, Sinha and Karn  \cite{SiKa} introduced the notion of (relatively) $p$-compact sets. The definition is inspired in Grothendieck's result which characterize relatively compact sets as those
contained in the convex hull of a norm null sequence of vectors of the space. In a similar form, $p$-compact sets are determined by norm $p$-summable sequences. Related to this concept, the ideal of  $p$-compact operators $\mathcal K_p$,  and different approximation properties naturally appear (see definitions
below). Since relatively $p$-compact  sets are, in particular, relatively compact, $p$-compact operators are compact. These concepts were first studied  in \cite{SiKa} and  thereafter in  several other articles, see for instance \cite{AMR, CK, DOPS, DPS_dens, DPS_adj, SiKa2008}. However, we believe that the goodness of the space of $p$-compact operators inherits from the general theory of operator ideals and tensor products has not yet been fully exploited.

The main purpose of this article is to show that the principal properties of  the class of $p$-compact operators  can be easily obtained if we study this operator ideal with the theory of tensor products and tensor norms. This insight allows us to give new results, to recover many already known facts, and also to improve some of them.  

The paper is organized as follows. In Section~\ref{sec:notation} we fix some notation and list the classical operator ideals, with their associated  tensor norms, which we use thereafter.  Section~\ref{sec:general-results} is devoted to general results on $p$-compact sets and $p$-compact operators.  We  define a measure $\m_p$, to study the size of a $p$-compact set $K$ in a Banach space $E$ and show that  this measure is independent of whether $K$ is considered as a subset of $E$ or as a subset of $E''$, the bidual of $E$. This allows us to show that  $\mathcal K_p$ is regular.  In addition, we prove that  $\mathcal K_p$ coincides isometrically with $\mathcal {QN}_p^{dual}$, the dual ideal of the quasi $p$-nuclear operators.  
We also show that any $p$-compact operator factors via a $p$-compact operator and two other compact operators. 

In Section~\ref{sec:tensor_norms} we use the Chevet-Saphar tensor norm $d_p$  to find the appropriate tensor norm associated to the ideal of $p$-compact operators.  We show that $\mathcal K_p$ is associated to the left injective  associate of  $d_p$, denoted by $/d_p$, which is equal to $g_{p'}'$.  We use this to link the theory of $p$-summing operators with that of $p$-compact operators. With the results known for the former class and natural hypothesis on $E$ and $F$ we show that $\mathcal K_p(E;F)$ and  $\mathcal K_q(E;F)$ coincide for a wide range of values of $p$ and $q$.  We also use the limit orders of the ideals of $p$-summing operators to show that our results are sharp. Furthermore, we prove that  $\mathcal K_p$ is  surjective, totally accessible and characterize its maximal hull $\mathcal K_p^{max}$ as the dual ideal of the $p$-summing operators, $\Pi_p^{dual}$.

For the sake of completeness, we list as an Appendix the limit orders of the ideal $p$-compact operators obtained by a simple
transcription of those given in \cite{PIE} for $p$-summing operators. 
\vskip .3cm 

 When the final version of this manuscript was being written, we got to know  a preprint  on the same subject authored by Albrecht Pietsch \cite{Pie2}. The main results in both articles coincide. However, the material in each paper was obtained independently. While A. Pietsch based his work using the classical theory of operator ideals following his monograph \cite{PIE}, we preferred the language of tensor products developed in the book by A. Defant and K. Floret \cite{DF}. 
 
\noindent {\bf Acknowledgement.} The authors would like to thank Daniel Carando for helpful comments and suggestions.

\section{Notation and Preliminaries}
\label{sec:notation}

Along this paper $E$ and $F$ denote Banach spaces, $E'$  and $B_E$ denote respectively the topological dual and the closed unit ball of $E$. A sequence  $(x_n)_{n}$ in $E$ is said to be $p$-summable if $\sum_{n=1}^{\infty} \|x_n\|^p<\infty$ and $(x_n)_{n}$ is said to be weakly $p$-summable if $\sum_{n=1}^{\infty} |x'(x_n)|^p < \infty$ for all $x' \in E'$. We denote, respectively,  $\ell_p(E)$  and $\ell_p^w(E)$ the spaces of
all $p$-summable and weakly $p$-summable sequences of $E$, $1\le p <\infty$. Both spaces are Banach spaces, the first one endowed with the norm $\|(x_n)_n\|_p=(\sum_{n=1}^{\infty} \|x_n\|^p)^{1/p}$ and the second if the norm $\|(x_n)_n\|_p^w = \d\sup_{x' \in B_{E'}}  {\{(\textstyle \sum_{n=1}^{\infty}} |x'(x_n)|^p)^{1/p}\}$ is considered. For $p=\infty$, we have the spaces $c_0(E)$ and $c_0^w(E)$ formed, respectively, by all null and weakly null sequences of $E$, endowed with the natural norms. The $p$-convex hull of a sequence $(x_n)_n$ in $\ell_p(E)$ is defined as $\pco\{x_n\}=\{\sum_{n=1}^{\infty} \alpha_n x_n \colon (\alpha_n)_n \in B_{\ell_{p'}}\}$ where $\frac1p+\frac1{p'}=1$ ($\ell_{p'}=c_0$ if $p=1$).

Following \cite{SiKa}, we say that a subset $K\subset E$ is relatively $p$-compact, $1\le p\le \infty$,  if there exists a sequence $(x_n)_n\subset \ell_p(E)$ so that $K\subset \pco\{x_n\}$.

The space of linear bounded operators from $E$ to $F$ is denoted by $\mathcal L(E;F)$ and its subspace of finite rank operators by $\mathcal F(E;F)$. Often the finite rank operator $x\mapsto \sum_{j=1}^{n} x'_j(x)y_j$ is associated with the element $\sum_{j=1}^n x_j'{\otimes} y_j$ in $E'\otimes F$.   In many cases, the completion of  $E'\otimes F$ with a reasonable tensor norms produces a subspace of  $\mathcal L(E;F)$. For instance  the injective tensor product $E'\hat\otimes_\eps F$ can be viewed as the approximable operators from $E$ to $F$. The Chevet-Saphar tensor norm $g_p$  defined on $E'\otimes F$ by $g_p(u)=\inf\{\|(x'_n)_n\|_p \|(y_n)_n\|^w_{p'} \colon u=\sum_{j=1}^n x'_j\otimes y_j\}$, gives the ideal of $p$-nuclear operators $\mathcal N_p(E;F)$, $1\le p \le \infty$. If we denote  $x'\underline{\otimes} y$ the 1-rank operator  $x\mapsto x'(x)y$, we have that
$$
\mathcal N_p(E;F)= \{T=\sum_{n=1}^{\infty} x'_n\underline{\otimes} y_n\colon (x'_n)_n\in \ell_{p}(E') \text{ and } (y_n)_n \in \ell^w_{p'}(F)\},
$$
is a Banach operator ideal endowed with the norm
$$
\v_p(T)=\inf\{\|(x'_n)_n\|_p \|(y_n)_n\|^w_{p'} \colon T=\sum_{n=1}^{\infty} x'_n\underline{\otimes} y_n\}.
$$
It is known that the space of $p$-nuclear operators is a quotient of  $E'\hat\otimes_{g_p} F$ and the equality $\mathcal N_p(E;F)=E'\hat\otimes_{g_p} F$ holds if either $E'$ or $F$ has the approximation property, see \cite[Chapter 6]{RYAN}. The definition of $g_p$ is not symmetric, its transpose $d_p=g_p^t$ is associated with  the operator ideal
$$
\mathcal N^p(E;F)=\{T=\sum_{n=1}^{\infty} x'_n\underline{\otimes} y_n\colon (x'_n)_n\in \ell^w_{p'}(E') \text{ and } (y_n)_n \in \ell_{p}(F)\},
$$
equipped with the norm
$$
\v^p(T)=\inf\{\|(x'_n)_n\|_{\ell^w_{p'}(E')} \|(y_n)_n\|_{\ell_p(F)} \colon T=\sum_{n=1}^{\infty} x'_n\underline{\otimes} y_n\}.
$$
Here, we have that $\mathcal N^p(E;F)=E'\hat\otimes_{d_p} F$ if either $E'$ or $F$ has the approximation property.  Also, note that when $p=1$, we obtain $\mathcal N_1=\mathcal N^1=\mathcal N$, the ideal of nuclear operators and $d_1=g_1=\pi$, the projective tensor norm.

In this paper, we are focused on the study of $p$-compact operators,  introduced  by Sinha and Karn \cite{SiKa} as those which map the closed unit ball into  $p$-compact sets. The space  of $p$-compact operators  is denoted by $\mathcal K_p(E;F)$, $1\le p\le \infty$ which is an operator Banach ideal endowed with the norm
$$
\kappa_p(T)=\inf\{\|(x_n)_n\|_p \colon T(B_E)\subset \pco\{x_n\}\}.
$$
We want to understand this operator ideal  in terms of a tensor product and a reasonable tensor norm. In order to do so we also make use of the ideal of the quasi $p$-nuclear operators introduced and studied by Persson and Pietsch \cite{PerPi}.   The space of quasi $p$-nuclear operators from $E$ to $F$ is denoted by $\mathcal {QN}_p(E;F)$. This ideal is associated by duality with the ideal of $p$-compact operators \cite{DPS_adj}. 

Recall that an operator $T$ is quasi $p$-nuclear if and only if there exists a sequence $(x'_n)_n\subset \ell_p(E')$, such that
$$
\|Tx\| \le \Big(\sum_n |x'_n(x)|^p\Big)^{\frac 1p},
$$
for all $x\in E$ and the quasi $p$-nuclear norm of $T$ is given by
$\v_p^Q(T)=\inf\{\|(x'_n)_n\|_p\},$ where the infimum is taken over all  the sequences $(x'_n)_n \in \ell_p(E')$ satisfying the inequality above. It is known that $\mathcal {QN}_p=\mathcal N_p^{inj}$, where $\mathcal N_p^{inj}$ denotes the injective hull of $\mathcal N_p$.

For  general background on tensor products and tensor norms we refer the reader to the monographs by Defant and Floret  \cite{DF} and by Ryan \cite{RYAN}. All the definitions and notation we use regarding tensor norms and operator ideals can be found in \cite{DF}. For further reading on operator ideals we refer the reader to Pietsch's book \cite{PIE}.

\section{On $p$-compact sets and $p$-compact operators}
\label{sec:general-results}

Given a relatively $p$-compact set $K$ in a Banach space $E$ there exists a sequence $(x_n)_n\subset E$ so that
$K\subset \pco\{x_n\}$. Such a sequence is not unique, then we may consider the following definition.

\begin{definition} Let $E$ be a Banach space, $K\subset E$ a $p$-compact set, $1\leq p \le \infty$. We define
$$
\m_p(K;E)=\inf \{\|(x_n)_n\|_{\ell_p(E)} \colon K\subset \pco \{x_n\}\}.
$$
If $K\subset E$ is not a $p$-compact set,  $\m_p(K;E)=\infty$. 
\end{definition}

We say that $\m_p(K;E)$ measures the size of $K$ as a $p$-compact set of $E$.

There are some properties which derive directly from the definition of $\m_p$. For instance, since $\pco \{x_n\}$ is absolutely convex, $ \m_p(K;E)=\m_p(\overline{\co} \{K\};E)$. Also, by H\"older's inequality, we have that $\|x\|\le \|(x_n)_n\|_{\ell_p(E)}$ and in consequence, $\|x\| \le \m_p(K)$, for all $x\in K$. Moreover,  as compact sets can be considered $p$-compact sets for $p=\infty$ we have that any $p$-compact set is $q$-compact  and  $\sup_{x \in K} \|x\|=\m_{\infty}(K;E)\leq \m_q(K;E)\leq\m_p(K;E)$, for $1\leq p\le q\leq \infty$.

Some other properties are less obvious. Suppose that $E$ is a subspace of another Banach space $F$.  It is clear that if $K\subset E$ is $p$-compact in $E$ then $K$ is $p$-compact in $F$ and
$\m_p(K;F)\le  \m_p(K;E)$. As we  see in Section~\ref{sec:tensor_norms}, the definition of $\m_p$ depends on the space $E$. In other words, $K$ may be $p$-compact in $F$ but not  in $E$.  We show this in Corollary~\ref{mp-cambia}.

For the particular case when $F=E''$, the bidual of $E$, Delgado, Pi\~neiro and Serrano  \cite[Corollary~3.6]{DPS_adj} show that  a set $K\subset E$ is $p$-compact if only if $K$ is $p$-compact in $E''$  with $\m_p(K;E'')\leq \m_p(K;E)$. We want to prove that, in fact, the equality $\m_p(K;E'')= \m_p(K;E)$ holds. In order to do so we propose to inspect various results concerning operators and their adjoint and show that the transpose operator is not only continuous but also an isometry.

Recall that when $E'$ has the approximation property, any operator $T \in \mathcal L(E;F)$ with nuclear adjoint $T'$, is nuclear and both nuclear norms coincide,  $\v(T)=\v(T')$, see for instance  \cite[Proposition 4.10]{RYAN}. The analogous result for $p$-nuclear operators is due to Reinov \cite[Theorem 1]{REI} and states that
when $E'$ has the approximation property and $T'$ belongs to $\mathcal N_p(F';E')$, then $T\in \mathcal N^p(E;F)$. However, the relationship between $\v^p(T)$ and $\v_p(T')$ is omitted. It is clear that whenever $T$ is in $\mathcal N^p(E;F)$  its adjoint is $p$-nuclear and, in that case, $ \v_p(T')\le \v^p(T)$. The Proposition~\ref{T'pnuc} below shows that the isometric result is also valid for $p$-nuclear operators.  Before showing this, we need the following result.

\begin{proposition}\label{reinov igualdad}
Let $E$ and $F$ be Banach spaces, $E'$ with the approximation property, and let $T\in \mathcal L(E;F)$. If $J_F   T \in \mathcal N^p(E;F'')$ then $T\in \mathcal N^p(E;F)$ and $\v^p(J_F  T)=\v^p(T)$.
\end{proposition}

\begin{proof} We only need to show the equality of the norms, the first part of the assertion corresponds with the first statement of  \cite[Theorem~1]{REI}. Note that since $E'$ has the approximation property, then $\mathcal N^p(E;F)=E'\hat{\otimes}_{d_p} F$ and $\mathcal N^p(E;F'')=E'\hat{\otimes}_{d_p} F''$.By the embedding lemma $E'\hat{\otimes}_{d_p}F$ is a subspace of $E'\hat{\otimes}_{d_p} F''$ via $Id_{E'} \otimes J_F$. Therefore, 
$$
\nu^p(J_F  T)=\nu^p(T),
$$
and the proof is complete. 
\end{proof}

\begin{proposition}\label{T'pnuc}
Let $E$ be a Banach space such that $E'$ has the approximation property and let $1\leq p <\infty$. If $T\in \mathcal L(E;F)$  has $p$-nuclear adjoint, then $T \in \mathcal N^p(E;F)$ and $\v^p(T)=\v_p(T')$.
\end{proposition}

\begin{proof} The first part of the assertion is a direct consequence of \cite[Theorem 1]{REI}. We only give  the proof which shows the isometric result. Take $T$ as in the statement. Since $T'$ belongs to $\mathcal N_p(F';E')$, there exist sequences $(y''_n)_n \in \ell_p(F'')$ and $(x'_n)_n \in \ell^w_{p'}(E')$ such that $T'=\sum_{n=1}^{\infty} y''_n \underline{\otimes}x'_n$. Then, $J_F  T= T''  J_E=\sum_{n=1}^{\infty} x'_n\underline{\otimes}y''_n$, which implies that $J_F  T\in \mathcal N^p(E;F'')$. It is clear that $\v_p(T')\geq \v^p(J_F  T)$.
By Proposition~\ref{reinov igualdad} we have that $T\in \mathcal N^p(E;F)$ and 
$\v^p(J_F  T)=\v^p(T)$. The reverse inequality always holds, whence the result follows.
\end{proof}

Now we are ready to prove that the $\m_p$-measure of a $p$-compact set $K\subset E$ does not change if $K$ is considered as a subset of $E''$. 

\begin{theorem}\label{K en E''}
Let $E$ be a Banach space and $K\subset E$. Then $K$ is $p$-compact in $E$ if and only if $K$ is $p$-compact in $E''$ and $\m_p(K;E)=\m_p(K;E'')$.
\end{theorem}
\begin{proof} We only need to show the inequality $\m_p(K;E)\leq \m_p(K;E'')$ since the claim  $K$ is $p$-compact in $E$ if and only if $K$ is $p$-compact in $E''$ is proved in \cite[Corollary~3.6]{DPS_adj}. Also, in this case,  the inequality  $\m_p(K;E'')\leq \m_p(K;E)$ is obvious.

Suppose that $K\subset E$ is $p$-compact and define the operator $\Psi\colon \ell_1(K)\rightarrow E$ such that for $\alpha=(\alpha_x)_{x\in K}$,
$$
\Psi(\alpha)=\sum_{x \in K}\alpha_x x.
$$
Note that $K\subset \Psi(B_{\ell_1(K)})\subset \overline{\co}(K)$ thus, $\Psi$ and $J_E  \Psi$ are $p$-compact operators. Also,  $\m_p(K;E)=\kappa_p(\Psi)$ and $\m_p(K;E'')=\kappa_p(J_E \Psi)$.
By \cite[Proposition ~3.1]{DPS_adj}, $\Psi' J_E'$ belongs to $\mathcal{QN}_p(E''';\ell_{\infty}(K))$ and $\v_p^Q(\Psi'  J_E')\le \kappa_p(J_E  \Psi)$. Therefore $\Psi'$ belongs to $\mathcal{QN}_p(E';\ell_{\infty}(K))$ and $\v_p^Q(\Psi')\le \v_p^Q(\Psi'  J_E')$.

Since $\ell_{\infty}(K)$ is injective, $\Psi ' \in \mathcal N_p(E';\ell_{\infty}(K))$ and $\v_p(\Psi')= \v_p^Q(\Psi')$, see \cite[Satz 38]{PerPi}. Now, an application of Proposition~\ref{T'pnuc}, gives us that $\Psi$ is an operator in $\mathcal N^p(\ell_1(K);E)$ and $\v^p(\Psi)=\v_p(\Psi')$. In particular, $\Psi \in \K_p(\ell_1(K);E)$ and $\kappa_p(\Psi)\leq \v^p(\Psi)$.

Thus, we have
$$
\m_p(K;E)=\kappa_p(\Psi)\leq \v^p(\Psi)=\v_p(\Psi')=\v_p^Q(\Psi')\le \v_p^Q(\Psi'  J_E')\le \kappa_p(J_E  \Psi)=\m_p(K;E''),
$$
and the proof is complete.
\end{proof}

\begin{theorem}
The operator ideal $\mathcal K_p$ is regular.
 \end{theorem}

\begin{proof} Let $E$ and $F$ be Banach spaces and $T\colon E\to F$ be an operator such that $J_FT$ is $p$-compact. Therefore, by the theorem above, $\m_p(J_FT(B_E);F'')= \m_p(T(B_E);F)$ and $T$ is $p$-compact. Whence, the result follows.
\end{proof}

Also we obtain the isometric version of \cite[Corollary~3.6]{DPS_adj} which is stated as follows.

\begin{corollary}\label{T'' p-compact} Let $E$ and $F$ be Banach spaces. Then $T\in \K_p(E;F)$ if and only if $T''\in \K_p(E'';F'')$ and $\kappa_p(T)=\kappa_p(T'')$.
\end{corollary}

\begin{proof} The statement $T\in \K_p(E;F)$ if and only if $T''\in \K_p(E'';F'')$ is part of \cite[Corollary~3.6]{DPS_adj}. Let $T$ be a $p$-compact operator. In particular, $T(B_E)$ is relatively compact and
$$
J_F  T(B_E)\subset T''(B_{E''})\subset \overline{J_F  T(B_E)}^{w^*} = \overline{J_F  T(B_E)}.
$$
Applying twice Theorem~\ref{K en E''} we get
$$
\m_p(T(B_E);F)=\m_p(T(B_E);F'')\leq \m_p(T''(B_{E''});F'')\leq \m_p(\overline{J_F  T(B_E)};F'')=\m_p(T(B_E);F).
$$
Since $\kappa_p(T)=\m_p(T(B_E);F)$ and $\kappa_p(T'')= \m_p(T''(B_{E''}); F'')$, the isometry  is proved.
\end{proof}

Now, we describe the duality between $p$-compact and quasi $p$-nuclear operators. On the one hand, we have that an operator  $T$   is quasi $p$-nuclear if and only if $T'$ is $p$-compact and $\kappa_p(T')= \v_p^Q(T)$ \cite[Corollary~3.4]{DPS_adj}. On the other hand, $T$ is  $p$-compact  if and only if its adjoint $T'$ is quasi $p$-nuclear, in this case $\v_p^Q(T')\le \kappa_p(T)$ \cite[Proposition~3.8]{DPS_adj}.  We improve this last result by showing that the transposition is an isometry.

\begin{corollary}\label{T Cuasi} Let $E$ and $F$ be Banach spaces. Then $T \in \mathcal K_p(E;F)$ if and only if $T'\in \mathcal {QN}_p(F';E')$ and $\kappa_p(T)=\v_p^Q(T')$.
\end{corollary}

\begin{proof} The inequality $\v_p^Q(T')\le \kappa_p(T)$ and the equality $\kappa_p(T'')= \v_p^Q(T')$ always hold.  A direct application of Corollary~\ref{T'' p-compact} completes the proof.
\end{proof}

With Corollary~\ref{T Cuasi} and  the results mentioned above we can state the following identities.

\begin{theorem} \label{Kpdual=QNp} $\mathcal K_p^{dual}\overset 1 = \mathcal {QN}_p$\ and\ \  $\mathcal {QN}_p^{dual}\overset 1=\mathcal K_p$.
\end{theorem}

We finish this section with a factorization result of $p$-compact operators, which improves \cite[Theorem 3.2]{SiKa} and \cite[Theorem~3.1]{CK}. The characterization given below should be compared with \cite[Proposition 5.23]{djt}. 

\begin{proposition}\label{factor}  Let $E$ and $F$ be Banach spaces. Then an operator  $T \in \mathcal L(E;F)$ is $p$-compact if and only if $T$ admits a factorization via a $p$-compact operator $T_0$ and a two compact operators $R$ and $S$ such that $T=ST_0R$.

Moreover, $\kappa_p(T)=\inf\{\|S\|\kappa_p(T_0)\|R\|\}$ where the infimum is taken over all the factorizations as above.
\end{proposition}

\begin{proof} Suppose that $T$ belongs to $\mathcal K_p(E;F)$. Then, given $\eps>0$, there exists $y=(y_n)_n \in \ell_p(F)$ such that $T(B_E)\subset \pco\{y_n\}$, with $\|(y_n)_n\|_p \le \kappa_p(T)(1+\eps)$. We may choose $\beta=(\beta_n)_n \in B_{c_0}$ such that $(\frac{y_n}{\beta_n})_n \in \ell_p(F)$ and $\|(\frac{y_n}{\beta_n})_n\|_p \leq \|(y_n)_n\|_p(1+\eps)$. Now, with $z=(z_n)_n=(\frac{y_n}{\beta_n})_n$, $T(B_E)\subset \{\sum_{n=1}^{\infty} \alpha_n z_n \colon (\alpha_n)_n \in L\}$ where $L$ is a compact set in $B_{\ell_{p'}}$. By the factorization given in  \cite[Theorem 3.2]{SiKa}, we have the commutative diagram
$$
\xymatrix{
E \ar[r]^{T} \ar[rd]_{R} &  F  & \ell_{p'} \ar[l]_{\theta_z} \ar[dl]^{\pi}  \\
& \ell_{p'}/ \ker \theta_z  \ar[u]_{\tilde \theta_z}  &
}
$$
where $\pi$ is the projection mapping and, $\theta_z$ and $R$ are given by  $\theta_z((\alpha_n)_n)=\sum_{n=1}^{\infty} \alpha_n z_n$ and  $R(x)=[(\alpha_n)_n]$ where $(\alpha_n)_n\in L$ is a
sequence satisfying that $T(x)=\sum_{n=1}^{\infty} \alpha_n z_n$.  Since $R(B_E)=\pi(L)$, we see that $R$  is  compact and $T=\tilde \theta_z R$.

Note also that $\tilde \theta_z$  is $p$-compact. Since $\|R\|\le 1$, then 
$$
\kappa_p(T)\le \kappa_p(\tilde \theta_z) \le \|(z_n)_n\|_p \le \kappa_p(T)(1+\eps)^2.
$$

Now, using \cite[Theorem 3.1]{CK} we factorize  $\tilde \theta_z$ via a $p$-compact operator $T_0$ and a compact operator $S$,  as follows:
$$
\xymatrix{
\ell_{p'} / \ker\theta_z   \ar[rr]^{\tilde \theta_z}  \ar[dr]_{T_0}  & & F \\
  & \ell_1/M  \ar[ur]_{_{S}} &
} 
$$
where $M$ is a closed subspace of $\ell_1$. A close inspection to the proof given in \cite{CK} allows us to chose the a factorization such that $\kappa_p(\tilde \theta_z) \le \|S\| \kappa_p(T_0) \le (1 +\eps) \kappa_p(\tilde \theta_z)$, (just consider 
a sequence $(\beta_n)_n$ similar to that used above).  Whence, the factorization  is obtained together with the desired equality  $\kappa_p(T)=\inf\{\|S\|\kappa_p(T_0)\|R\|\}$.

The reverse claim is obvious.
\end{proof}

Note that if both $E'$ and $F$ have the approximation property  then $T$ belongs to $\mathcal K_p(E;F)$ if and only if $T$ belongs to  $\mathcal K_p^{min}(E;F)$. Moreover, $\kappa_p(T)=\kappa_p^{min}(T)$. We show in the next section that the same result holds if only one of the spaces ($E'$ or $F$) has the approximation property.

\section{Tensor norms}
\label{sec:tensor_norms}

Our purpose in this section is  to draw together the theory of operator ideals and tensor products for the class of $p$-compact operators. To start with our aim we use the Chevet-Saphar tensor norm to find the appropriate tensor norm associated to the ideal of $p$-compact operators.  The tensor norm obtained is $g_{p'}'$ which allows us to connect the theory of $p$-summing operators with that of $p$-compact operators. With the results known for the former class, with additional hypothesis on $E$ and $F$ we show that $\mathcal K_p(E;F)$ and  $\mathcal K_q(E;F)$ coincide for a wide range of $p$ and $q$.  We also use the limit orders of the ideal of $p$-summing operators \cite{PIE} to show that the values considered for $p$ and $q$  cannot be improved. Some other properties describing the structure of the ideal of $p$-compact operators are given.

Recall that $d_p(u)=\inf\{\|(x_n)_n\|^w_{p'} \|(y_n)_n\|_{p}\}$ where the infimum is taken over all the possible representations of  $u=\sum_{j=1}^n x_j\otimes y_j$. We denote by $/d_p$ the left injective tensor norm associated to $d_p$. Note  that $/d_p=g_{p'}'$ \cite[Theorem~7.20]{RYAN} and therefore $/d_p=(g_{p'}^*)^t$.

\begin{proposition} \label{surjective ideal}The ideal $(\mathcal K_p, \kappa_p)$ is surjective.
 \end{proposition}

\begin{proof} Let $Q\colon G\overset 1 \twoheadrightarrow  E$ be a quotient map, if $T Q$ is $p$-compact, then $TQ(B_G)=T(B_E)$ is a $p$-compact set. Thus, $T$ is $p$-compact and
$$
\begin{array}{lr}
\kappa_p(TQ)=\m_p(TQ(B_G))=\m_p(T(B_E))=\kappa_p(T).
 \end{array}
$$
\end{proof}

In order to  characterize the tensor norm associated to  $(\mathcal K_p,\kappa_p)$ we need the following simple lemma. We sketch its proof for completeness. This result should be compared with \cite[Theorem 20.11]{DF}.

\begin{lemma} \label{lema: alpha inj} Let  $(\mathcal A, \|.\|_{\mathcal A})$ be an operator ideal and let $\alpha$ be its associated tensor norm.
\begin{enumerate}
 \item[\rm (a)] If $\mathcal A$ is surjective then, $\alpha$ is left injective.
 \item[\rm (b)] If $\mathcal A$ is injective then, $\alpha$ is right injective.
\end{enumerate}
\end{lemma}

\begin{proof} Suppose $\mathcal A$ is surjective. Using a `left version' of  \cite[Proposition 20.3 (1)]{DF},  we only need to see that $\alpha$ is left injective on  $FIN$, the class of all finite dimensional spaces.

Fix $N, M, W\in FIN$ such that $i\colon M \overset 1 \hookrightarrow W$, then we have the commutative diagram
$$
\xymatrix{
& M\otimes_\alpha N \ar[r]^{i\otimes id_N} \ar@{=}[d]  & W\otimes_\alpha N \ar@{=}[d]  \\
 & \mathcal A(M';N) \ar[r]_{\phi}   &
\mathcal A(W';N)
}
$$
where $\phi$ is given by $T\mapsto Ti'$. As $i$ is an isometry, $i'$ is a metric surjection. Now, since $\mathcal A$ is   surjective $\phi$ is an isometry, which proves (a).

The proof of (b) follows easily with a similar reasoning.
\end{proof}

From \cite[Proposition 3.11]{DPS_adj}  we have  $\mathcal N^p(\ell_1^n; N)\overset 1 = \mathcal K_p(\ell_1^n; N)$, for every $n$ and every finite dimensional space $N$. Since  $\mathcal N^p$ is associated to the tensor norm $d_p$, we have the following result.

\begin{theorem}\label{/dp}
The operator ideal $(\mathcal K_p,\kappa_p)$ is associated to the tensor norm $/d_p$, for every $1\leq p < \infty$.
\end{theorem}
\begin{proof} Denote by $\alpha$ the tensor norm associated to $\mathcal K_p$.  By Proposition~\ref{surjective ideal} and the above lemma,  $\alpha$ is left injective. Note that for every $n$ and every finite dimensional space $N$  we have the isometric identities
$$
\ell_\infty^n \otimes_{d_p} N  = \mathcal N^p(\ell_1^n; N)  = \mathcal K_p(\ell_1^n; N)  =\ell_\infty^n \otimes_{\alpha} N.
$$
Now, applying a  `left version' of  \cite[Proposition 20.9]{DF}, we conclude that $\alpha=/d_p$.
\end{proof}

\begin{proposition}\label{injective}
The operator ideal $(\mathcal K_p,\kappa_p)$ is not injective, for any $1\leq p < \infty$.
\end{proposition}

\begin{proof} Suppose that $\mathcal K_p$ is injective. By Theorem~\ref{/dp} and Lemma~\ref{lema: alpha inj}  we see that $/d_p$  the associated tensor norm to  $\mathcal K_p$, is right injective. Thus, its transpose $g_{p'}^*$  is left injective. Now, by  \cite[Theorem 20.11]{DF}, $\Pi_p$ is surjective which is a contradiction. Note that, by Grothendieck's theorem \cite[Theorem 23.10]{DF}, $id\colon \ell_2 \to\ell_2$ belongs to $\Pi_p^{sur}$ and obviously   is not $p$-summing.
\end{proof}

As a consequence we show that the $\m_p$-measure of a set, depends on the space which contains the set.

\begin{corollary}\label{mp-cambia} Given $1\le p<\infty$, there exist a Banach space $G$, a subspace $F\subset G$ and a set $K\subset F$ such that $K$ is $p$-compact in $G$ but $K$ fails to be $p$-compact in $F$.
\end{corollary}

\begin{proof}
Since $(\mathcal K_p,\kappa_p)$ in not injective, there exist Banach spaces $E, F$ and $G$, $F\stackrel{I_{F,G}}\hookrightarrow G$ and an operator $T\in \mathcal L(E;F)$ such that $I_{F,G} T$ is $p$-compact but $T$ is not. Taking $K=T(B_E)$, we see that  $\m_p(K;G)<\infty$ while $\m_p(K;F)=\infty$.
\end{proof}

Now we characterize $\mathcal K_p^{max}$, the maximal  hull of the operator ideal $\mathcal K_p$ in terms of the ideal of  $p$-summing operators $\Pi_p$.

\begin{corollary}\label{maxhull} The operator  ideal $\mathcal K_p^{max}$ coincides isometrically with  $\Pi_p^{dual}$.
 \end{corollary}

\begin{proof} The maximal hull of  $\mathcal K_p$ is also associated to the tensor norm $/d_p=(g_{p'}^*)^t$. Since the ideal of $p$-summing operators $\Pi_p$ is associated to the tensor norm $g_{p'}^*$, by Corollary~3~in~\cite[17.8]{DF} the result follows.
\end{proof}

By \cite[Proposition 21.1 (3)]{DF} and the fact that $/(d_p/) = /d_p$ we see that the tensor norm  $/d_p$ is totally accessible (see also \cite[Corollary~7.15]{RYAN}).  Therefore,  we have the following two results. For the first one we use \cite[Proposition~21.3]{DF} and for the second one we use \cite[Corollary~22.2]{DF}.

\begin{remark} \label{remark grossa}
The operator ideal $\mathcal{K}_p^{max} \overset 1 = \Pi_p^{dual}$ is totally accessible.
\end{remark}

\begin{remark} \label{remark kp-min} For any Banach spaces $E$ and $F$,
$\mathcal{K}_p^{min}(E;F) \overset 1 = E'\widehat \otimes_{/d_p} F$.
\end{remark}

With the help of Corollary~\ref{maxhull} we obtain an easy way to compute the $\kappa_p$ norm of a $p$-compact operator: just take the $p$-summing norm of its adjoint. Moreover, the same holds for the minimal norm. 
We also have the following isometric relations. 

\begin{proposition}\label{inclusiones isometricas} The isometric inclusions hold
$$
\xymatrix{
\mathcal K_p^{min} \ar@{^{(}->}[r]^{1} & \mathcal K_p \ar@{^{(}->}[r]^{1\;\;\;\;\;\;\;\;\;\;\;\;\;\;} &  {\mathcal K_p^{max} \overset 1 = \Pi_p^{dual}.}
}
$$
In particular, $\mathcal K_p^{min}$ and $\mathcal K_p$ are totally accessible. 
\end{proposition}

\begin{proof} Let $E$ and $F$ be Banach spaces. We have
$$
\xymatrix{
\mathcal K_p^{min}(E;F) \ar@{^{(}->}[r]^{\;\;\; \leq \; 1} & \mathcal K_p(E;F) \ar@{^{(}->}[r]^{\; \leq \; 1  \; \; \; \; \; \; \; \; \; \; \; \; \; \; \; \; \; \; \; \;} &  {\mathcal K_p^{max}(E;F) \overset 1 = \Pi_p^{dual}(E;F).}
}
$$
Now, using the previous remark and \cite[Corollary~22.5]{DF}, we have
$\mathcal K_p^{min}(E;F) \overset{1}\hookrightarrow  \mathcal K_p^{max}(E;F)\overset 1 =  \Pi_p^{dual}(E;F)$, which implies that  all the inclusions above are isometries.
\end{proof}

The definition of the $\kappa_p$-approximation property was given in terms of operators in \cite{DPS_dens}:  A Banach space  $F$ has the $\kappa_p$-approximation property if, for every Banach space $E$, $\mathcal{F}(E;F)$ is $\kappa_p$-dense in $\mathcal K_p(E;F)$.  In other words,
$$\overline{\mathcal F(E;F)}^{\kappa_p} \overset 1 = \mathcal K_p(E;F).$$
On the other hand, by Remark \ref{remark grossa}, \cite[Corollary~22.5]{DF} and the previous proposition we have
$$ \mathcal K_p^{min}(E;F) \overset 1 = \overline{\mathcal F(E;F)}^{\mathcal K_p^{max}} \overset 1 = \overline{\mathcal F(E;F)}^{\kappa_p}.
$$
Therefore, $F$ has the $\kappa_p$-approximation property if and only if $\mathcal K_p^{min}(E;F) \overset 1 = \mathcal K_p (E;F)$, for every Banach space $E$.

Any Banach space with the approximation property enjoys the $\kappa_p$-approximation property. This result can be deduced from \cite[Theorem~3.1]{DPS_dens}. Below, we give a short proof using the language of operator ideals. It is worthwhile mentioning that every Banach space has the $\kappa_2$-approximation property (which can be deduced from \cite[Theorem 6.4]{SiKa}) and for each $p\ne 2$ there exists a Banach space whose dual lacks the $\kappa_p$-approximation property \cite[Theorem~2.4]{DPS_dens}.

\begin{proposition}\label{pa implica kp-pa}  If a Banach space has the approximation property  then it has the $\kappa_p$-approximation property.
\end{proposition}

\begin{proof}
We have shown that a Banach space  $F$ has the $\kappa_p$-approximation property if and only if $\mathcal K_p^{min}(E;F) \overset 1 = \mathcal K_p (E;F)$, for every Banach space $E$.
Suppose that $F$ has the approximation property and let $T \in \mathcal K_p (E;F)$. Using \cite[Theorem~3.1]{CK} we have the following factorization
$$
\xymatrix{
E \ar[rr]^{T} \ar[rd]_{T_0} & & F   \\
& G  \ar[ur]_{S} &
},
$$
where $T_0$ is $p$-compact and $S$ is compact (therefore approximable).
Now, by  \cite[Proposition~25.2 (1) b.]{DF}, $T$ belongs to $\mathcal K_p^{min}(E;F)$, which concludes the proof.
\end{proof}

Note that, in general,  the converse of Proposition~\ref{pa implica kp-pa},  is not true. For instance, if $1\le p<2$, we always may find a subspace $E \subset \ell_q$, $1< q<2$ without the approximation property. This subspace is reflexive and has cotype 2. Then, by the comment bellow  \cite[Proposition~21.7]{DF} and  applying \cite[Corollary~2.5]{DPS_adj} to obtain that $F=E'$ has the $\kappa_p$-approximation property and it fails to have the approximation property.

In this setting, the next theorem becomes quite natural. It states that the ideal of $p$-compact operators can be represented in terms of tensor products in presence of the $\kappa_p$-approximation property.

\begin{theorem}\label{Kp como tensor} Let $E$ and $F$ be Banach spaces. Then,
$$
E'\widehat \otimes_{/d_p}F \overset 1 = \mathcal K_p(E;F)
$$
if and only if $F$ has the $\kappa_p$-approximation property.
Also, the isometry remains valid whenever $E'$ has the approximation property, regardless of $F$.
\end{theorem}

\begin{proof} Note that $/ d_p$ is totally accessible (see the comments preceding Remark \ref{remark grossa}). Thus, the proof of the first claim is straightforward from Remark~\ref{remark kp-min}.

For the second statement, take $T\in  \mathcal K_p(E;F)$. By Proposition~\ref{factor}, $T=T_0R$ where $R$ is a compact operator and $T_0$ is $p$-compact. Now, using that $E'$ has the approximation property, $R$ is approximable by finite rank operators and an application of   \cite[Proposition~25.2 (2) b.]{DF} gives that $T\in  \mathcal K_p^{min}(E;F)$. Again, the result follows by  Remark~\ref{remark kp-min}.
\end{proof}

The next result improves \cite[Proposition~3.3]{DPS_dens}.

\begin{corollary} Let $E$ and $F$ be  Banach spaces such that $F$ has the $\kappa_p$-approximation property or $E'$ has the approximation property . Then, $\mathcal K_p(E;F)'\overset 1 = \mathcal I_{p'}(E';F')$, $1\le p\le \infty$.
\end{corollary}

\begin{proof} The proof is straightforward from  Theorem~\ref{Kp como tensor} and \cite[Pag 174]{RYAN}.
\end{proof}

Now, we compare $p$-compact and $q$-compact operators for certain classes of Banach spaces. We use some well known results stated for $p$-summing operators when the spaces involved are of finite cotype or  $\mathcal L_{q,\lambda}$-spaces, for some $q$.  Our results are stated in terms of $\mathcal K_p^{min}(E;F)$ but if $F$ has the $\kappa_p$-approximation property or $E'$ has the approximation property, by Theorem~\ref{Kp como tensor}, they  can be stated for $\mathcal K_p(E;F)$. First we need the following general result. As usual, for $s=\infty$, we consider $\mathcal L(X;Y)$ instead of $\Pi_s(X;Y)$ and $\overline{\mathcal F(Y;X)}$ instead of  $\mathcal K_s^{min}(Y;X)$.

\begin{theorem}\label{mega-teo} Let $X$ and $Y$ be Banach spaces such that for some $1\le r < s \le\infty$   $\Pi_r(X';Y')=\Pi_s(X';Y')$.  Then, $\mathcal K_s^{min}(Y;X)=\mathcal K_r^{min}(Y;X)$.

Moreover, if  $\pi_r (\cdot) \leq A \pi_s(\cdot)$ on $\Pi_s(X';Y')$  then $\kappa_r(\cdot) \leq A\kappa_s(\cdot)$ on $K_s^{min}(Y;X)$, $A>0$.
\end{theorem}

\begin{proof} Suppose that $\Pi_r(X';Y')=\Pi_s(X';Y')$. Since $\Pi_r$ is a maximal ideal and its associated tensor norms, $g_{r'}^*$ is totally accessible \cite[Corollary 21.1.]{DF} we have, by the Embedding Theorem \cite[17.6.]{DF}, $X''\widehat \otimes_{g_{r'}^*} Y' \overset 1\hookrightarrow \Pi_r(X';Y')$. Now, using the Embedding lemma \cite[13.3.]{DF} we have the following commutative diagram

$$
\xymatrix{
Y'\widehat \otimes_{/d_s} X = X\widehat \otimes_{g_{s'}^*} Y' \ar@{^{(}->}[r]^{\;\;\;\;\;\;\;\;\;\;\;\;1} \ar@{-->}[d]^{\leq \; A} & X''\widehat \otimes_{g_{s'}^*} Y' \ar[r]^{1} \ar@{^{(}->}[r]^{1} & \Pi_s(X';Y') \ar@{^{(}->}[d]^{\leq \; A}\\
Y'\widehat \otimes_{/d_r} X = X\widehat \otimes_{g_{r'}^*} Y' \ar@{^{(}->}[r]^{\;\;\;\;\;\;\;\;\;\;\;\;1 } & X''\widehat \otimes_{g_{r'}^*} Y'  \ar@{^{(}->}[r]^{1} & \Pi_r(X';Y')
}.
$$

Therefore, $/d_s \leq /d_r \leq A \;  /d_s$ on $Y' \otimes X$, which implies that $\mathcal K_s^{min}(Y;X)=\mathcal K_r^{min}(Y;X)$ and $\kappa_r(T) \leq A\kappa_s(T)$ for every $T \in K_s^{min}(Y;X)$.
\end{proof}

In order to compare the norm $\kappa_r(T)$ with $\|T\|$ or with $\kappa_s(T)$, we use the constants obtained in comparing  summing operators, taken from \cite{TJ}. Some of them  involve the Grothendieck constant $K_G$, the constant $B_r$  taken from Khintchine's inequality and $C_q(E)$ the $q$-cotype constant of $E$. With this notation and the theorem above we have the following results.

\begin{corollary}\label{l2 l1} Let $E$ and $F$ be Banach spaces such that $E$ is a   $\mathcal L_{2,\lambda'}$-space and $F$ is a  $\mathcal L_{\infty,\lambda}$-space. Then,  $\overline{\mathcal F(E;F)} = \mathcal K_1^{min}(E;F)$ and $\kappa_1(T) \leq K_G \lambda \lambda' \|T\|$ for every $T \in \overline{\mathcal F(E;F)}$.
\end{corollary}

\begin{proof} Note that $E$ is a  $\mathcal L_{2,\lambda'}$-space if and only if $E'$ is a  $\mathcal L_{2,\lambda'}$-space and $F$ is a  $\mathcal L_{\infty,\lambda}$-space if and only if $F'$ is a  $\mathcal L_{1,\lambda}$-space, see \cite[23.2 Corollary 1]{DF} and \cite[23.3]{DF}.  Now, use  Theorem~\ref{mega-teo} with  \cite[Theorem 23.10]{DF} or \cite[Theorem 10.11]{TJ}.
\end{proof}

\begin{corollary}\label{linfty} Let $E$ and $F$ be Banach spaces such that $F$ is a $\mathcal L_{1,\lambda}$-space. Then,
 \begin{enumerate}
\item[\rm (a)]  if $E'$ has cotype 2, $\overline{\mathcal F(E;F)} = \mathcal K_2(E;F) = \mathcal K_r^{min}(E;F)$, for all $2\le r$ and 
$$
\kappa_r(T)\le \lambda\left[c\, C_2(E')^2\left(1+\log C_2(E')\right)\right]^{1/r}\|T\|,
$$
for all $T \in \mathcal K_r^{min}(E;F)$.

\item[\rm (b)]  if $E'$ has cotype $q$, $2< q < \infty$, $\overline{\mathcal F(E;F)} = \mathcal K_r^{min}(E;F)$ for all $q<r<\infty$ and
$$
\kappa_r(T)\leq \lambda \ c \ q^{-1}(1/q-1/r)^{-1/r'} C_q(E') \|T\|,
$$
for all $T \in \mathcal K_r^{min}(E;F)$.
\end{enumerate}
In each case, $c>0$ is a universal constant.
\end{corollary}

\begin{proof} Again,  $F$ is a  $\mathcal L_{\infty,\lambda}$-space if and only if $F'$ is a  $\mathcal L_{1,\lambda}$-space. For the first statement, note that every space has the $\kappa_2$-approximation property,  $\mathcal K_2^{min}(E;F)= \mathcal K_2(E;F)$. Now, use Theorem~\ref{mega-teo} with the combination of
\cite[Theorem 10.14]{TJ} and \cite[Proposition 10.16]{TJ}. For the second claim, use  Theorem~\ref{mega-teo} and \cite[Theorem 21.4 (ii)]{TJ}.
\end{proof}

\begin{corollary}\label{nohipo} Let $E$ and $F$ be Banach spaces.  Then,
 \begin{enumerate}
\item[\rm (a)] if $E'$ has cotype $2$, $\mathcal K_r^{min}(E;F) = \mathcal K_2(E;F)$, for all $2\le r < \infty$  and $$
\kappa_2(T) \leq B_r C_2(E') \kappa_r(T),
$$ 
for every $T \in \mathcal K_r^{min}(E;F)$. 

\item[\rm (b)]  if $F'$ has cotype $2$,  $\mathcal K_2^{min}(E;F) = \mathcal K_1^{min}(E;F)$, for all $E$ and
$$
\kappa_1(T)\leq c\, C_2(F') (1+\log C_2(F'))^{1/2}\kappa_2(T),
$$
for every $T \in \mathcal K_2^{min}(E;F)$. 

\noindent In particular, for all $1\le r \le 2$, $\mathcal K_r^{min}(E;F) = \mathcal K_1^{min}(E;F)$, for all $E$. 

\item[\rm (c)]  if $F'$ has cotype $q$, $2< q < \infty$, $\mathcal K_r^{min}(E;F) = \mathcal K_1^{min}(E;F)$, for all $1\le r < q'$ and for all $E$, and
$$
\kappa_1(T)\leq c\, q^{-1}(1/q-1/r')^{-1/r} C_q(F') \kappa_r(T),
$$
for every $T \in \mathcal K_r^{min}(E;F)$.
\end{enumerate}

In each statement $c>0$ is a universal constant.

Note that if $E'$ and $F'$ have cotype 2,  $\mathcal K_r^{min}(E;F) = \mathcal K_1^{min}(E;F)$, for all $1\le r < \infty$.
\end{corollary}

\begin{proof} Use  Theorem~\ref{mega-teo} and 
\cite[Theorem~10.15]{TJ} for part (a).  For (b) use
\cite[Corollary~10.18 (a)]{TJ}. Finally, use  Theorem~\ref{mega-teo} with 
\cite[Corollary~21.5 (i)]{TJ} for the third claim.
\end{proof}

We finish this section by showing that the conditions considered on $r$ in the corollaries above  are sharp. We make use of the notion of limit order \cite[Chapter 14]{PIE}, which has proved useful, specially to compare different operator ideals. Recall that for an operator ideal $\mathcal A$,  the limit order $\lambda(\mathcal A, u, v)$ is defined to be the infimum of all $\lambda\ge 0$ such that the diagonal operator $D_\lambda$ belongs to $\mathcal{A}(\ell_u; \ell_v)$, where $D_\lambda\colon (a_n)\mapsto (n^{-\lambda}a_n)$ and $1\le u,v\le \infty$. 

\begin{lemma}\label{limit_order} Let  $1\le u,v,p \le \infty$ and $u', v', p'$ the respective conjugates. Then,
$$
\lambda (\mathcal K_p, u,v) = \lambda (\Pi_p, v', u').
$$
\end{lemma}

\begin{proof} Denote by $id_{u,v}$ the identity map from $\ell_u^n$ to $\ell_v^n$, for a fixed integer $n$. By Corollary~\ref{maxhull} we have
$$ \kappa_p(id_{u,v} \colon \ell_u^n\to \ell_v^n)  = \pi_p(id_{v',u'}  \colon \ell_{v'}^n\to \ell_{u'}^n).$$
Then, a direct application of \cite[Theorem~14.4.3]{PIE} gives the result.
\end{proof}

Now we have:

\begin{result}  The conditions imposed on $r$ in Corollary~\ref{linfty} and  Corollary~\ref{nohipo} are sharp.
\end{result}

\begin{enumerate}
\item Let $E=\ell_u$ and $F=\ell_1$. Note that (see Appendix (a) and (b))
$$
\lambda (\mathcal K_r, u,1) =
\begin{cases}
 1 -\frac 1u & \quad \text{if}\quad  r'\le u\le\infty,\\
\frac 1r   & \quad \text{if}\quad   \quad 1\le u\le r'.
\end{cases}
$$
Fixed $1\le r  < 2$ choose $2< u <r'$, then $E'$ has cotype 2 and  $\lambda (\mathcal K_r, u,1) =\frac 1r \neq \frac 1{u'} = \lambda (\mathcal K_2, u,1)$.  Thus,  $\mathcal K_r(\ell_u;\ell_1) \ne \mathcal K_{2}(\ell_u;\ell_1)$ and $r$ cannot be included in Corollary~\ref{linfty} (a), whenever $1\le r<2$. 
\smallskip

\noindent Now, fix  $2<q$ and let $E=\ell_{q'}$. Then $E'$ has cotype $q$ and given $r<q$,  we see that $\lambda (\mathcal K_r, q',1) = \frac 1r$. On the other hand, $\lambda (\mathcal K_s, q',1) = \frac 1q$ for any $q<s$. This shows that $\mathcal K_r(\ell_{q'};\ell_1)\neq \mathcal K_{s}(\ell_{q'};\ell_1)$ for any $r< q<s$. 

\noindent Note that we have also shown that if $r <\tilde r\le q$, then   $\lambda (\mathcal K_{\tilde r}, q',1) \ne \lambda (\mathcal K_r, q',1)$. Therefore, the inclusions $\mathcal K_{\tilde r}(\ell_{q'},\ell_1) \subset \mathcal K_{r}(\ell_{q'},\ell_1) $ are always strict, for any  $r <\tilde r\le q$.
\vskip .3cm

\noindent For the case $r=q$, $2< q < \infty$, take $E=L_{q'}[0,1]=L_{q'}$ and $F=L_1[0,1]=L_1$. Suppose that $\overline{\mathcal F(L_{q'};L_1)} = \mathcal K_q(L_{q'};L_1)$. By Theorem~\ref{Kp como tensor},  $L_q\widehat  \otimes_{/d_q}L_1=L_q\widehat \otimes_{\eps}L_1$ and $L_1\widehat  \otimes_{(/d_q)^t}L_q=L_1\widehat \otimes_{\eps}L_q$.  Since $(/d_q)^t=(d_{q'})'$, and $\pi'=\eps$, then $L_1\widehat \otimes_{d_{q'}'}L_q=L_1\widehat\otimes_{\pi'}L_q$ and 
we get that $(L_1\widehat \otimes_{d_{q'}'}L_q)'=(L_1\widehat\otimes_{\pi'}L_q)'$. Since both $L_\infty$ and $L_{q'}$  have the metric approximation property, by \cite[17.7]{DF} and \cite[12.4]{DF}, we have the isomorphism  $L_{\infty}\widehat\otimes_{d_{q'}}L_{q'}=L_{\infty}\widehat\otimes_{\pi}L_{q'}$. Therefore $(L_{\infty}\widehat\otimes_{d_{q'}}L_{q'})'=(L_{\infty}\widehat\otimes_{\pi}L_{q'})'$. In other words, $\Pi_q(L_{\infty},L_q)=\mathcal L(L_{\infty},L_q)$ (see \cite[Section 6.3]{RYAN}). This last equality contradicts~\cite[Theorem 7]{Kwa}.  
\smallskip

\item For any $1\leq p < \infty$, there exits a compact operator  in $\mathcal L(\ell_p;\ell_p)$ (and therefore approximable), which is not $p$-compact \cite[Example 3.1]{AMR}. Thus, $\overline{\mathcal F(\ell_p;\ell_p)}\ne \mathcal K_p (\ell_p;\ell_p)$.

\noindent Fix $p\geq 2$,  for all $2\le r < \infty$, we see that 
$\mathcal K_r^{min} (\ell_p;\ell_p)=  \mathcal K_p^{min} (\ell_p;\ell_p) =  \mathcal K_p (\ell_p;\ell_p) = \mathcal K_2 (\ell_p;\ell_p)$,  Corollary~\ref{nohipo}~(a).  Then, $r=\infty$ cannot be included in the first statement of this corollary. 
\smallskip

\noindent Also, for $r< 2$, we may choose $p$ and $q$ such that $2\leq p \leq r'$ and $1\leq q\leq r$. Now, with 
$E=\ell_p$ and $F=\ell_q$ we compute the limit orders  (see Appendix) obtaining $\lambda (\mathcal K_r, p,q)= \frac 1r$ and $\lambda (\mathcal K_2, p,q)= \frac 12$, we conclude that the inclusion  $\mathcal K_r(\ell_p;\ell_q)\subset \mathcal K_2(\ell_p;\ell_q)$ is strict.
\smallskip

\item To see that the choice of $r$ in Corollary~\ref{nohipo} (b) is sharp, fix $r>2$. Take $p$ and $q$ such that $2\leq q <r$ and $1\leq p \leq r'$. Let $E=\ell_p$ and $F=\ell_{q'}$, using the limit orders we obtain  $\lambda (\mathcal K_2, p,{q'})= \frac 12$ and $\lambda (\mathcal K_r, p,{q'})= \frac 1r$ (see Appendix (b)). Thus, $\mathcal K_2(\ell_p;\ell_{q'})\neq\mathcal K_r(\ell_p;\ell_{q'})$.

\noindent Here, we  have also shown that if $2\le r <\tilde r$,  the inclusions $\mathcal K_{\tilde r}(\ell_p;\ell_{q'}) \subset \mathcal K_{r}(\ell_p;\ell_{q'}) $ are strict, for suitable $p$ and $q$.
\smallskip

\item Now, we focus our attention on Corollary~\ref{nohipo} (c). Fix $2<q$ and let $E=\ell_1$ and $F=\ell_{q'}$. We claim that $\mathcal K_r(\ell_1,\ell_{q'}) \neq \mathcal K_1(\ell_1, \ell_{q'})$, for any $q'<r$.  In fact, the result follows using the limit orders: $\lambda (\mathcal K_1, 1,{q'})= \frac 1{q'}$ and $\lambda (\mathcal K_r, 1,{q'})= \frac 1r$. This also shows that 
$\mathcal K_{\tilde r}(\ell_1;\ell_{q'})$ is strictly contained in  $\mathcal K_{r}(\ell_1;\ell_{q'})$  for any $q'\le r < \tilde r$.
\smallskip

\noindent  Finally, we deal with the remaining case, $r=q'$. Take $E=L_1[0,1] =L_{1}$,  $F=L_{q'}[0,1] =L_{q'}$ and suppose that $\mathcal K_{q'}(L_1;L_{q'})=\mathcal K_{1}(L_1;L_{q'})$,  $2<q<\infty$. Applying Theorem~\ref{Kp como tensor} we get that $L_{\infty}\widehat\otimes_{g'_q}L_{q'}=L_{\infty}\widehat\otimes_{g'_{\infty}}L_{q'}$. Thus, the tensor spaces  have isomorphic duals.  By \cite[17.7]{DF} and \cite[13.3]{DF} we obtain the isomorphism $L_1\widehat\otimes_{g_q}L_q=L_1\widehat\otimes_{g_{\infty}}L_q$. 
Since $g_{\infty}= \backslash\eps$ and $g_q=  \backslash g_{q'}^{*}$, by \cite[Corollary 1 20.6]{DF}, $L_1\widehat\otimes_{g_{q'}^*}L_q=L_1\widehat\otimes_{\eps}L_q$. As shown in part (1), this cannot happen.
\end{enumerate}

\section{Appendix}
In this section we  transcribe the  values of  $\lambda (\Pi_r, v', u')$, computed in  Pietsch's monograph, to give $\lambda(\mathcal K_r, u, v)$. To this end, we combine Propositions 22.4.9, 22.4.12 and 22.4.13 in \cite{PIE}.

\begin{enumerate}
\item[\rm (a)] For $1\le r\le 2$,
$$
\lambda (\mathcal K_r, u,v) =
\begin{cases}
\frac 1r & \quad \text{if}\quad  1\le v \le r, \quad 1\le u\le r',\\
1 -\frac 1u & \quad \text{if}\quad  1\le v \le r, \quad r'\le u\le\infty,\\
\frac 1v   & \quad \text{if}\quad  r\le v \le 2, \quad 1\le u\le v',\\
1 -\frac 1u & \quad \text{if}\quad  r\le v \le 2, \quad v'\le u\le\infty,\\
\frac 1v   & \quad \text{if}\quad   2\le v \le \infty,  \quad 1\le u\le 2,\\
\frac 12  - \frac1u +\frac1v& \quad \text{if}\quad  2\le v \le \infty, \quad 2\le u\le\infty.\\
\end{cases}
$$

\item[\rm (b)] For $2 <r < \infty$,
$$
\lambda (\mathcal K_r, u,v) =
\begin{cases}
\frac 1r & \quad \text{if}\quad  1\le v \le r, \quad 1\le u\le r',\\
1-\frac 1u   & \quad \text{if}\quad  1\le v\le 2, \quad r'\leq u \leq \infty,\\
\rho   & \quad \text{if}\quad  2\le v \le r, \quad r'\le u\le 2,\\
\frac 1v & \quad \text{if}\quad  r\le v \le \infty, \quad 1\le u\le2,\\
\frac 12  - \frac1u +\frac1v& \quad \text{if}\quad  2\le v\le\infty,\quad 2\leq u \leq \infty,\\
\end{cases}
$$

\noindent where $\d \rho=\frac 1r + \frac {(\frac 1v -\frac 1r)(\frac 1{r'}- \frac 1u)}{\frac 12 - \frac 1r}$.
\end{enumerate}


\begin{thebibliography}{99}
\bibitem{AMR} Aron R., Maestre M., Rueda P. \textit{$p$-compact holomorphic mappings}, RACSAM {\bf 104} (2) (2010), 353--364.



\bibitem{CK} Choi Y. S., Kim J. M.  \textit{The dual space of $(\mathcal L(X,Y);\tau_p)$ and the p-approximation property},  J. Funct. Anal {\textbf{259}}, (2010) 2437-2454.


\bibitem{DF} Defant, A. Floret, K. Tensor norms and operators ideal, North Holland Publishing Co., Amsterdam, 1993.

\bibitem{DOPS} Delgado, J. M., Oja, E., Pi\~neiro, C., Serrano, E. \textit{The p-approximation property in terms of density of finite rank operators}, J. Math Anal, Appl. \textbf{354} (2009), 159-164.

\bibitem{DPS_dens} Delgado, J. M., Pi\~neiro, C., Serrano, E. \textit{Density of finite rank operators in the Banach space of $p$-compact operators}, J. Math. Anal. Appl. \textbf{370} (2010), 498-505.

\bibitem{DPS_adj} Delgado, J. M., Pi\~neiro, C., Serrano, E. \textit{Operators whose adjoints are quasi $p$-nuclear}, Studia Math. \textbf{197} no. 3 (2010), 291-304.

\bibitem{djt} Diestel, J., Jarchow, H., Tonge, A. Absolutely summing operators. Cambridge Studies in Advanced Mathematics, 43. Cambridge University Press, Cambridge, 1995.






\bibitem{Gro} Grothendieck, A. \textit{R\'esum\'e de la th\'eorie m\'etrique des produits tensoriels topologiques}, Bol. Soc. Mat. S\~ao Paolo 8 (1956), 1-79.

\bibitem{Kwa} Kwapie\'{n}, S. \textit{ On a theorem of L. Schwartz and its applications to absolutely summing operators}, Studia Math. \textbf{38} (1970), 193-201.
 


\bibitem{REI} Reinov, O. \textit{On linear operators with p-nuclear adjoint}, Vestnik St. Petersburg Univ. Math 33 (2000), no. 4, 19-21.

\bibitem{PIE} Pietsch, A. Operators ideals, North Holland Publishing Company, Amsterdam/New York/Oxford, 1980.

\bibitem{Pie2} Pietsch, A.  \textit{The ideal of p-compact operators and its maximal hull}. Preprint.

\bibitem{PerPi} Persson, A., Pietsch, A. {\it
$p$-nukleare une $p$-integrale Abbildungen in Banachr\"aumen}, (German)
Studia Math. 33 1969 19--62.

\bibitem{RYAN} Ryan, R. Introduction to Tensor products on Banach Spaces, Springer, London, 2002.


\bibitem{SiKa} Sinha, D. P., Karn, A. K. \textit{Compact operators whose adjoints factor trough subspaces of $\ell_p$}, Studia Math. \textbf{150} (2002), 17--33.

\bibitem{SiKa2008} Sinha, D. P., Karn, A. K. \textit{Compact operators which factor trough subspaces of $\ell_p$}, Math. Nachr. \textbf{281} (2008), 412--423.

\bibitem{TJ} Tomczak-Jaegermann, N. Banach-Mazur Distances and Finite-Dimensional Operator Ideals,
Longman Scientific and Technical, Harlow, 1989.
\end{thebibliography}
\end{document}